\documentclass[11pt]{article}
\usepackage{amssymb}
\usepackage{amsthm}
\usepackage{amsmath}
\usepackage{graphicx}

 \newtheorem{thm}{Theorem}[section]
 \newtheorem{cor}[thm]{Corollary}
 \newtheorem{lem}[thm]{Lemma}
 
 \theoremstyle{definition}
 \newtheorem{defn}[thm]{Definition}
 \newtheorem{rem}[thm]{Remark}
 \numberwithin{equation}{section}

\theoremstyle{definition}

\theoremstyle{remark}

\begin{document}

\title{Criticality of the Axially Symmetric \\
Navier-Stokes Equations}

\author{ Zhen Lei\footnote{School of Mathematics, Fudan University, Shanghai 200433, P.
 R. China.}\and Qi S. Zhang\footnote{Department of Mathematics,  University of California,
Riverside, CA 92521} }

\date{}

\maketitle

\begin{abstract}
Smooth solutions to the axi-symmetric Navier-Stokes equations obey the following maximum principle: $$\sup_{t\geq 0}\|rv^\theta(t, \cdot)\|_{L^\infty} \leq \|rv^\theta(0, \cdot)\|_{L^\infty}.$$ We prove that all solutions with initial data in $H^{\frac{1}{2}}$ is smooth globally in time if $rv^\theta$ satisfies a kind of Form Boundedness Condition (FBC) which is invariant under the natural scaling of the Navier-Stokes equations. In particular, if $rv^\theta$ satisfies
\begin{equation}\nonumber
\sup_{t \geq 0}|rv^\theta(t, r, z)| \leq C_\ast|\ln r|^{- 2},\ \ r \leq \delta_0 \in (0, \frac{1}{2}),\ C_\ast < \infty,
\end{equation}
then our FBC is satisfied. Here $\delta_0$ and $C_\ast$ are independent of neither 
the profile nor the norm of the initial data. So the gap from regularity is logarithmic in nature.
We also prove the global regularity of solutions if 
$\|rv^\theta(0, \cdot)\|_{L^\infty}$ or 
$\sup_{t \geq 0}\|rv^\theta(t, \cdot)\|_{L^\infty(r \leq r_0)}$ is 
small but the smallness depends on certain dimensionless quantity of the initial data.
\end{abstract}

\section{Introduction}

The global regularity problem of 3D incompressible Navier-Stokes equations is commonly considered as supercritical because the \textit{a priori} estimates based on energy equality become worse when looking into finer and finer scales (for instance, see Tao \cite{Tao}). Such a "super-
criticality" barrier is one of the main reasons why this problem is such a hard problem.

Recently, the axi-symmetric Navier-Stokes equations have attracted tremendous interests of experts. See, for instance, \cite{BZ:1, CL, CSTY1, CSTY2, CFZ, HL, HLL, JX, KNSS, LNZ, LZ, LZ-2, LMNP, NP, NP-2, Pan, SS, TX, ZZ} and many others. These results heavily depend on the maximum principle of the dimensionless quantity $\Gamma = rv^\theta$, which makes the axi-symmetric Naiver-Stokes equations partially critical (only the swirl component $v^\theta$ of the velocity field satisfies a dimensionless \textit{a priori} estimate). Although the axially symmetric Navier-Stokes equation is a special case of the full three-dimensional one,  our  level  of understanding  had  been  roughly  the  same,  with  essential
difficulty unresolved because the effective \textit{a priori} bound available is still the
energy estimate, which has a positive dimension $\frac{1}{2}$.   

The aim of this article is to show that the axi-symmetric Naiver-Stokes equation
 is in fact fully critical. More precisely, we prove that all solutions with initial data in $H^{\frac{1}{2}}$ are smooth globally in time if $rv^\theta$ satisfies a kind of Form Boundedness Condition (FBC) which is invariant under the natural scaling of the Navier-Stokes equations. In particular, if $rv^\theta$ satisfies
\begin{equation}\nonumber
\sup_{t \geq 0}|rv^\theta(t, r, z)| \leq C_\ast|\ln r|^{- 2},\ \ r \leq \delta_0 \in (0, \frac{1}{2}),\ C_\ast < \infty,
\end{equation}
then our FBC is satisfied. Here $\delta_0$ and $C_\ast$ are independent of neither the profile 
nor the norm of the initial data.  The proof is based on the observation that the vorticity equations 
can be transformed into a system such that the vortex stretching terms are critical. This means 
that the potentials in front of unknown functions scale as $1/|x|^2$. For example, in (\ref{eqJOmg})
below, the function $J$ is a regarded as unknown and the potential in front of it is $-2 \frac{v^{\theta}}{r}$
which scales as $1/|x|^2$.

 We also prove the global regularity of solutions if 
$\sup_{t \geq 0}\|rv^\theta(t, \cdot)\|_{L^\infty(r \leq r_0)}$ or $\|rv^\theta(0, \cdot)\|_{L^\infty}$  is small but the smallness depends on certain dimensionless quantity of the initial data. Our work is inspired by the recent interesting result of Chen-Fang-Zhang in \cite{CFZ} where, among other things, global regularity is obtained  if $rv^\theta(t, \cdot, z)$ is H${\rm \ddot{o}}$lder continuous in $r$ variable.

To state our result more precisely, let us recall that in cylindrical coordinates $r,\theta, z$ with
$(x_1,x_2,x_3)=(r\cos\theta,r\sin\theta,z)$, axially symmetric
solutions of the Navier-Stokes equations are of the following form
\begin{align*}
\begin{cases}
v(t, x)=v^r(t, r, z)e_r + v^{\theta}(t, r, z) e_{\theta} +v^z(t, r, z) e_z,\\[-4mm]\\
p(t, x) = p(t, r, z).
\end{cases}
\end{align*}
The components $v^r,v^{\theta},v^z$
 are all independent of the angle of rotation $\theta$.
 Here $e_r, e_{\theta}, e_z$ are the basis vectors for
 $\mathbb{R}^3$ given by
\begin{align*}
e_r = \Big (\frac{x_1}{r},\frac{x_2}{r},0 \Big )^\top, \ e_{\theta}=
\Big (\frac{-x_2}{r},\frac{x_1}{r},0 \Big )^\top, \ e_z
=(0,0,1)^\top.
\end{align*}
In terms of $(v^r, v^\theta, v^z, p)$, the axi-symmetric Navier-Stokes equations are
\begin{align}\label{ANS}
\begin{cases}
\partial_tv^r + (v^re_r + v^ze_z)\cdot\nabla v^r - \frac{(v^\theta)^2}{r} + \partial_rp = (\Delta - \frac{1}{r^2})v^r,\\[-4mm]\\
\partial_tv^\theta + (v^re_r + v^ze_z)\cdot\nabla v^\theta + \frac{v^rv^\theta}{r} = (\Delta - \frac{1}{r^2})v^\theta,\\[-4mm]\\
\partial_tv^z + (v^re_r + v^ze_z)\cdot\nabla v^z +  \partial_zp = \Delta v^z,\\[-4mm]\\
\partial_rv^r + \frac{v^r}{r} + \partial_zv^z = 0.
\end{cases}
\end{align}
It is well-known that finite energy smooth solutions of the Navier-Stokes equations satisfy the following energy identity
\begin{equation}\label{NI}
\|v(t, \cdot)\|_{L^2}^2 + 2\int_0^t\|\nabla v(s, \cdot)\|_{L^2}^2ds = \|v_0\|_{L^2}^2,\quad \forall\ \ t \geq 0.
\end{equation}

Denote
\begin{equation}\nonumber
\Gamma = rv^\theta.
\end{equation}
One can easily check that
\begin{equation}\label{Max}
\partial_t\Gamma + (v^re_r + v^ze_z)\cdot\nabla \Gamma = (\Delta - \frac{2}{r}\partial_r)\Gamma.
\end{equation}
A significant consequence of \eqref{Max} is that smooth solutions of the axi-symmetric Navier-Stokes equations satisfy the following maximum principle (see, for instance, \cite{CL, HL, CSTY1, NP, NP-2}):
\begin{equation}\label{MaxP}
\sup_t\|\Gamma(t, \cdot)\|_{L^\infty} \leq \|\Gamma_0\|_{L^\infty}.
\end{equation}
We emphasis that $\|\Gamma(t, \cdot)\|_{L^\infty}$ is a dimensionless quantity with respect to the natural scaling of the Navier-Stokes equations.  From this point of view, the axi-symmetric Naiver-Stokes equations can be seen as partially critical, while the general Navier-Stokes equations are known to be supercritical (see \cite{Tao}).

Now let us introduce the function class where $v^\theta$ lives. It is defined in an integral way which is usually called Form Boundedness Condition (FBC), which is in spirit similar to a condition so that certain Hardy type inequality holds.

\begin{defn}\label{FBC}
We say that the angular velocity $v^\theta(t, r, z)$ is in $(\delta_\ast, C_\ast)$-critical class if
\begin{equation}\label{FBC-1}
\int \frac{|v^\theta|}{r}|f|^2dx \leq C_\ast\int|\partial_rf|^2dx + C_0\int_{r \geq r_0}|f|^2dx,
\end{equation}
\begin{equation}\label{FBC-2}
\int |v^\theta|^2|f|^2dx \leq \delta_\ast\int|\partial_rf|^2dx + C_0\int_{r \geq r_0}|f|^2dx,
\end{equation}
holds for some $r_0 > 0$, some $C_0 > 0$ and for all $t \geq 0$ and all axi-symmetric scalar and vector functions $f \in H^1$.
\end{defn}
Clearly, under the natural scaling of the Navier-Stokes equations:
$$v^\lambda(t, x) = \lambda v(\lambda^2t, \lambda x),\quad p^\lambda(t, x) = \lambda^2 p(\lambda^2t, \lambda x),$$
the above definition of FBC is invariant: $(v^\lambda)^\theta$ also satisfies \eqref{FBC-1}-\eqref{FBC-2} if $v^\theta$ does.  Below is the first result of this article:
\begin{thm}\label{thm}
There exists a constant $\delta_\ast > 0$ such that for arbitrary $C_\ast > 1$, all local strong solutions to the axially symmetric Navier-Stokes equations with initial data $\|v_0\|_{H^{\frac{1}{2}}} < \infty$ and $\|\Gamma_0\|_{L^\infty} < \infty$, if the angular velocity field $v^\theta$ is in $(\delta_\ast, C_\ast)$-critical class, i.e. $v^\theta$ satisfies the critical Form Boundedness Condition in \eqref{FBC-1}-\eqref{FBC-2},
then $v$ is regular globally in time.
\end{thm}

An important corollary of Theorem \ref{thm} is:
\begin{cor}\label{cor}
Let $\delta_0 \in (0, \frac{1}{2})$ and $C_1 > 1$. Let $v$ be the local strong solution of the axially symmetric Navier-Stokes equations with initial data $v_0  \in H^{\frac{1}{2}}$ and $\|\Gamma_0\|_{L^\infty} < \infty$.
If
\begin{equation}\label{CD}
\sup_{0 \leq t < T}|\Gamma(t, r, z)| \leq C_1|\ln r|^{- 2},\ \ r \leq \delta_0,
\end{equation}
then $v$ is regular globally in time.
\end{cor}
We emphasis that $C_\ast$ in Theorem \ref{thm} and $C_1$ in Corollary \ref{cor} are independent of neither the profile nor the norm of the given initial data. The proof of this Corollary will be given at the end of  section 2. The point is that if \eqref{CD} is satisfied, then the FBC of \eqref{FBC-1}-\eqref{FBC-2} is true. Then one can apply Theorem \ref{thm} to get the desired conclusion.

Our work is inspired by a recent very interesting work by Chen, Fang and Zhang \cite{CFZ} where, among other things, the authors proved that $v$ is regular if $\Gamma$ is H${\rm \ddot{o}}$lder continuous. Let
\begin{equation}\nonumber
\Omega = \frac{\omega^\theta}{r},\quad J = - \frac{\partial_zv^\theta}{r}.
\end{equation}
We emphasis that $J$ was introduced by Chen-Fang-Zhang in \cite{CFZ}, while $\Omega$ appeared much earlier and can be at least tracked back to the book of Majda-Bertozzi in \cite{MB}. Both of the two new variables are of great importance in our work. Following \cite{MB, HL, CFZ},  we also study the equations for $J$ and $\Omega$:
\begin{align}
\label{eqJOmg}
\begin{cases}
\partial_tJ + (b\cdot\nabla)J =  \big (\Delta + \frac{2}{r}\partial_r \big
)J + (\omega^r\partial_r + \omega^z\partial_z)\frac{v^r}{r},\\[-4mm]\\
\partial_t\Omega + (b\cdot\nabla)\Omega =   \big  (\Delta + \frac{2}{r}\partial_r  \big
)\Omega - 2\frac{v^\theta}
{r}J.
\end{cases}
\end{align}
Here $\omega^\theta$ is the angular component of the vorticity $\omega = \nabla\times v$, which reads
 \begin{align*}
\omega(t, x) =\omega^re_r + \omega^\theta
e_{\theta} + \omega^ze_z,
\end{align*}
with
\begin{equation}\nonumber
\omega^r =  - \partial_z v^\theta,\quad \omega^{\theta} =\partial_z v^r - \partial_r
v^z,\quad \omega^z = \partial_r v_{\theta} + \frac{v^\theta}{r}.
\end{equation}
Our new observation is that the axi-symmetric Navier-Stokes equations exhibit certain critical nature when being formulated in terms of a new set of unknowns $J$ and $\Omega$.
Our second observation is that, with the FBC assumptions \eqref{FBC-1}-\eqref{FBC-2}, the stretching term $(\omega^r\partial_r + \omega^z\partial_z)\frac{v^r}{r}$ in the equation of $J$ could be arbitrarily small, by using the relation of $v^r, v^z$ and $\Omega$ in Lemma \ref{Lem1} (which was originally proved by Hou, Lei and Li \cite{HLL} in the periodic case and later on extended to the general case by Lei \cite{Lei-mhd}. Alternatively, one may also use the magic formula given by Miao and Zheng in \cite{MZ} to prove it). Then we can derive a closed \textit{a priori} estimate for $J$ and $\Omega$ using the first two observations and the structure of the stretching term in $\Omega$-equation.

Our second goal is to prove that the smallness of $\sup_{t \geq 0}\|\Gamma(t, \cdot)\|_{L^\infty(r \leq r_0)}$ or $\|\Gamma_0\|_{L^\infty}$ implies the global regularity of the solutions. Recently, Chen-Fang-Zhang in \cite{CFZ} proved that, among many other interesting results, if $\Gamma(t, \cdot, z)$ is H${\rm \ddot{o}}$lder continuous  in $r$-variable, then the solution of the axi-symmetric Navier-Stokes equations is smooth. Both of the results depend on given initial data. More precisely, the smallness of $\sup_{t \geq 0}\|\Gamma(t, \cdot)\|_{L^\infty(r \leq r_0)}$ or $\|\Gamma_0\|_{L^\infty}$  in our Theorem \ref{SmallnessRegu} depends on other dimensionless norms of the initial data. From this point of view, our result improves the one in \cite{CFZ}.

We define
$$V = \frac{v^\theta}{\sqrt{r}}.$$
Here is the second main result.
\begin{thm}\label{SmallnessRegu}
Let $r_0 > 0$. Suppose that $v_0 \in H^{\frac{1}{2}}$ such that
$\Omega_0 \in L^2$, $V_0^2 \in L^2$ and $\Gamma_0 \in L^2\cap
L^\infty$. Denote
\begin{equation}\nonumber
\big(\|\Omega_0\|_{L^2} + \|V_0^2\|_{L^2}\big)\|\Gamma_0\|_{L^2} = M_0
\end{equation}
and
$$\big(\|V_0^2\|_{L^2} +
  \|\Omega_0\|_{L^2} + r_0^{-2}\|v_0\|_{L^2}\|\Gamma_0\|_{L^\infty}^{\frac{3}{2}}\big)\|\Gamma_0\|_{L^2} = M_1.$$
There exists an absolute positive (small) constant $\delta > 0$ such
that if either
\begin{equation}\nonumber
\|\Gamma_0\|_{L^\infty}  \leq \delta M_0^{-1},
\end{equation}
or
\begin{equation}\nonumber
\sup_{t \geq 0}\|\Gamma(t, \cdot)\|_{L^\infty(r \leq r_0)}  \leq \delta M_1^{-1},
\end{equation}
then the axially symmetric Navier-Stokes equations are globally
well-posed.
\end{thm}

The proof of Theorem \ref{SmallnessRegu} is based on a new formulation of the axi-symmetric Navier-Stokes equations \eqref{ANS} in terms of  $V$ and $\Omega = \frac{\omega^\theta}{r}$, and also on the estimate of $\frac{v^r}{r}$ in terms of $\Omega$ and its derivative (see Lemma \ref{Lem1}).

Now let us recall some highlights on the study of the axi-symmetric Navier-Stokes equations.
If the swirl $v_\theta = 0$, then it
is known since the late 1960's (see Ladyzhenskaya \cite{L},
Uchoviskii and Yudovich \cite{UY}), that finite energy
solutions to (\ref{ANS}) are smooth for all time. See also the
paper by Leonardi, Malek, Necas and Pokorny \cite{LMNP}. In the presence of swirl,
it is not known in general if finite energy solutions blow up or not in finite time. In \cite{HL}, Hou and Li constructed a family of large solutions based on some deep insights on a 1D model. See also some extended results in \cite{HLL} by  Hou, Lei and Li. We also mention various \textit{a priori} estimates of smooth solutions by Chae and Lee \cite{CL} and Burke-Zhang \cite{BZ:1}. To the best of our knowledge, the best a priori bound of velocity field is given in \cite{LNZ}:
$$|v(t, x)| \leq C_\ast r^{- 2}|\ln r|^{\frac{1}{2}}.$$

In \cite{CSTY1}, Chen, Strain, Tsai and Yau obtained a lower
bound for the possible blow up rate of singularties: if
$$|v(t, x)| \leq \frac{C_\ast}{r},$$
then $v$ is regular. This seems to be the first time that people can exclude possible singularities in the presence of assumptions on $|x|^{-1}$ type non-smallness quantities. Soon later, Chen, Strain, Yau  and Tsai \cite{CSTY2} and Koch, Nadirashvili, Seregin  and Sverak \cite{KNSS} extended the result of \cite{CSTY1} and in particular, excluded the possibility of type I singularities of $v$. See also a local version by Seregin and  Sverak \cite{SS} and various extensions by Pan \cite{Pan}. We also mention that in \cite{LZ}, Lei and Zhang excluded the possibility of singularities under $v^re_r + v^ze_z \in L^\infty([0, T], {\rm BMO}^{-1})$ based on an observation in \cite{LZ-2}. This solves the regularity problem of $L^\infty([0, T], {\rm BMO}^{-1})$ solutions of Navier-Stokes equations in the axi-symmetric case. Moreover, it extends the result of \cite{CSTY1} and \cite{KNSS} since the assumptions on the axial component of velocity $|v^z| \leq C_\ast r^{-1}$ itself implies $v^re_r + v^ze_z \in L^\infty([0, T], {\rm BMO}^{-1})$ (see \cite{LNZ} for details).

Let us also mention that Neustupa and Pokorny \cite{NP}
proved that the regularity of one component (either $v^r$ or
$v^\theta$) implies regularity of the other components of the
solution. Also proving regularity is the work of Jiu  and   Xin
\cite{JX} under an assumption of sufficiently small zero-dimension
scaled norms.  See more refined results in \cite{NP-2} and the work of Ping Zhang and Ting Zhang \cite{ZZ}. Chae and  Lee \cite{CL} also proved regularity results
assuming finiteness of another certain zero-dimensional integral.  Tian and Xin \cite{TX} constructed a family of
singular axially symmetric solutions with singular initial data.

The remainder of the paper is simply organized as follows. In  section 2 we
recall two basic lemmas and prove Corollary \ref{cor} by assuming the validity of Theorem \ref{thm}. Then we prove Theorem \ref{thm} in section 3. Section 4 is devoted to the proof of Theorem \ref{SmallnessRegu}.

\section{Notations and Lemmas}

For abbreviation, we denote
$$b(t, x) = v^r e_r + v^ze_z.$$
The last equation in \eqref{ANS} shows that $b$ is divergence-free. The Laplacian operator $\Delta$ and the gradient operator $\nabla$ in the cylindrical coordinate
are
 \begin{align*}
 \Delta = \partial^2_r +
\frac{1}{r}\partial_r + \frac{1}{r^2}\partial^2_\theta
 + \partial^2_z,\quad
 \nabla = e_r\partial_r +\frac{e_\theta}{r}\partial_\theta + e_z\partial_z.
\end{align*}
For scalar axi-symmetric function $f(t, r, z)$, we often use the commutation property:
$$\nabla\partial_rf(r, z) = \partial_r\nabla f(r, z).$$
Throughout the proof, we will denote
$$\|f\|_{L^2}^2 = \int|f|^2rdrdz,\quad dx = rdrdz.$$

The following estimate will be used very often. It was originally proved in \cite{HLL} in the periodic case and then extended to the general case in \cite{Lei-mhd} (see (4.5)-(4.6) there by noting the relation $v^r = - \partial_z\psi^\theta$). Alternatively, one may also use the magic formula given by Miao and Zheng in \cite{MZ} to prove it.
\begin{lem}\label{Lem1}
Let $v^r$ be the radial component of the velocity field and $\Omega = \omega^\theta/r$. Then there exists an absolute positive constant $K_0 > 0$ such that
\begin{equation}\nonumber
\|\nabla\frac{v^r}{r}\|_{L^2} \leq K_0\|\Omega\|_{L^2},\quad \|\nabla^2\frac{v^r}{r}\|_{L^2} \leq K_0\|\partial_z\Omega\|_{L^2}.
\end{equation}
\end{lem}

The following lemma gives the uniform decay estimate for the angular component of vorticity in
$r$ direction for large $r$.  We point out that a weaker estimate for
$\omega^\theta$ has appeared in \cite{CL}. Even though we don't need to use  the estimate for $\omega^r$ and $\omega^z$ in this paper, we will still include them below for possible future use.
\begin{lem}\label{decay}
Suppose that $v_0 \in L^2$ is an axially symmetric divergence-free
vector and $(r\omega^r_0, r^2\omega^\theta_0, r\omega^z_0) \in L^2$.
Then the smooth solution of the Navier-Stokes equation with initial
data $v_0$ satisfies the following \textit{a priori} estimates:
\begin{eqnarray}\label{vorticity-1}
&&\sup_{0 \leq t < T}\big(\|r\omega^r(t, \cdot)\|_{L^2}^2,
  \|r\omega^z(t, \cdot)\|_{L^2}^2\big) + \int_0^T
  \big(\|\nabla[r\omega^r(t, \cdot)]\|_{L^2}^2,
  \|\nabla[r\omega^z(t, \cdot)]\|_{L^2}^2\big)dt\\\nonumber
&&\quad\quad\quad\quad \leq \|r\omega^r_0\|_{L^2}^2 + \|r\omega^z_0\|_{L^2}^2 +
4(\|\Gamma_0\|_{L^\infty}^2 + 1)\|v_0\|_{L^2}^2,
\end{eqnarray}
and
\begin{eqnarray}\label{vorticity-2}
&&\quad\quad\quad\quad\quad \sup_{0 \leq t < T}\|r^2\omega^\theta(t, \cdot)\|_{L^2}^2 +
  \int_0^T\|\nabla(r^2\omega^\theta)\|_{L^2}^2dt\\\nonumber
&&\leq C_0\Big(\|r^2\omega^\theta_0\|_{L^2}^2 + \big(
  \|v_0\|_{L^2}^4 + \|\Gamma_0\|_{L^3}^2\big)\|v_0\|_{L^2}^2\Big)
  \exp\Big\{\frac{T}{\|\Gamma_0\|_{L^3}^2 +
  \|v_0\|_{L^2}^4}\Big\},
\end{eqnarray}
where $C_0$ is a generic positive constant.
\end{lem}
\begin{proof}
First of all, let us recall that
\begin{equation}\label{NS-vorticity-2}
\begin{cases}
\partial_t\omega^r + b\cdot\nabla \omega^r - \partial_rv^r\omega^r =
  \big(\Delta - \frac{1}{r^2}\big)\omega^r + \partial_zv^r\omega^z,\\[-4mm]\\
\partial_t\omega^\theta + b\cdot\nabla \omega^\theta - \frac{v^r}{r}\omega^\theta =
  \big(\Delta - \frac{1}{r^2}\big)\omega^\theta + \partial_z\frac{(v^\theta)^2}{r},\\[-4mm]\\
\partial_t\omega^z + b\cdot\nabla \omega^z - \partial_zv^z\omega^z =
\Delta \omega^z + \partial_rv^z\omega^r.
\end{cases}
\end{equation}
Let us first prove \eqref{vorticity-1}. Take the $L^2$ inner product
of $\omega^r$ equation with $r^2\omega^r$ and $\omega^z$ equation
with $r^2\omega^z$ in \eqref{NS-vorticity-2}, we have
\begin{eqnarray}\nonumber
&&\frac{1}{2}\frac{d}{dt}\int(\omega^r)^2r^2dx - \int r^2
  \omega^r\big(\Delta - \frac{1}{r^2}\big)\omega^r dx\\\nonumber
&&= - \int r^2\omega^rb\cdot\nabla \omega^rdx + \int r^2
  \partial_rv^r(\omega^r)^2dx + \int
  r^2\omega^r\partial_zv^r\omega^zdx
\end{eqnarray}
and
\begin{eqnarray}\nonumber
&&\frac{1}{2}\frac{d}{dt}\int(\omega^z)^2r^2dx - \int r^2
  \omega^z\Delta\omega^z dx\\\nonumber
&&= - \int r^2\omega^zb\cdot\nabla \omega^zdx + \int r^2
  \partial_zv^z(\omega^z)^2dx + \int
  r^2\omega^z\partial_rv^z\omega^rdx
\end{eqnarray}
Using integration by parts, we have
\begin{eqnarray}\nonumber
- \int r^2\omega^r\big(\Delta - \frac{1}{r^2}\big)\omega^r
  dx = \int|\nabla (r\omega^r)|^2dx
\end{eqnarray}
and
\begin{eqnarray}\nonumber
- \int r^2\omega^z\Delta\omega^zdx = \int|\nabla (r\omega^z)|^2dx -
2\int|\omega^z|^2dx.
\end{eqnarray}
Using the integration by parts and the incompressibility constraint, one has
\begin{eqnarray}\nonumber
&&- \int r^2\omega^rb\cdot\nabla \omega^rdx - \int r^2
  \omega^zb\cdot\nabla \omega^zdx\\\nonumber
&&\quad +\ \int r^2\partial_rv^r(\omega^r)^2dx + \int r^2
  \partial_zv^z(\omega^z)^2dx\\\nonumber
&&= \int (rv^r + r^2\partial_rv^r)(\omega^r)^2 + (rv^r +
  r^2\partial_zv^z)(\omega^z)^2 dx\\\nonumber
&&= - \int[\partial_zv^z(r\omega^r)^2 + \partial_rv^r
  (r\omega^z)^2]dx.
\end{eqnarray}

Consequently, we have
\begin{eqnarray}\nonumber
&&\frac{1}{2}\frac{d}{dt}\int[(\omega^r)^2 + (\omega^z)^2]r^2dx +
  \int \big(|\nabla (r\omega^r)|^2 + |\nabla
  (r\omega^z)|^2\big)dx\\\nonumber
&&= \int|\omega^z|^2dx  - \int[\partial_zv^z(r\omega^r)^2 +
  \partial_rv^r(r\omega^z)^2]dx\\\nonumber
&&\quad +\ \int \big(r^2\omega^r\partial_zv^r\omega^z + r^2
  \omega^z\partial_rv^z\omega^r\big)dx\\\nonumber
&&\leq 2\int|\omega^z|^2dx  + \|\nabla b\|_{L^2}
  \big(\|r\omega^r\|_{L^4}^2 + \|r\omega^z\|_{L^4}^2\big).
\end{eqnarray}
Note that by Gagliardo-Nirenberg's inequality and the maximum
principle $\|\Gamma\|_{L^\infty} \leq \|\Gamma_0\|_{L^\infty}$, one has
\begin{eqnarray}\nonumber
&&\|r\omega^r\|_{L^4}^2 + \|r\omega^z\|_{L^4}^2 =
  \|\nabla\Gamma\|_{L^4}^2\\\nonumber
&&\quad\quad\quad = \big(\int- \Gamma\nabla\cdot(\nabla\Gamma |\nabla\Gamma|^2)dx\big)^{\frac{1}{2}} \leq 3\|\Gamma\|_{L^\infty}\|\Delta\Gamma\|_{L^2}\\\nonumber
&&\quad\quad\quad\leq 3\|\Gamma_0\|_{L^\infty}\big(\|\partial_r(r\omega^z)\|_{L^2}
  + \|\partial_z(r\omega^r)\|_{L^2} + \|\omega^z\|_{L^2}\big).
\end{eqnarray}
Hence, by Holder inequality, we have
\begin{eqnarray}\nonumber
&&\frac{d}{dt}\int[(\omega^r)^2 + (\omega^z)^2]r^2dx +
  \int \big(|\nabla (r\omega^r)|^2 + |\nabla
  (r\omega^z)|^2\big)dx\\\nonumber
&&\quad\quad\quad\quad\quad\quad\leq 4(\|\Gamma_0\|_{L^\infty}^2 + 1)\|\nabla
  b\|_{L^2}^2.
\end{eqnarray}
Integrating the above differential inequality with respect to time
and recalling the basic energy estimate, one gets
\eqref{vorticity-1}.

Next, let us prove \eqref{vorticity-2}. Let us first write the
equation of $\omega^\theta$ in \eqref{NS-vorticity-2} as:
\begin{equation}\nonumber
\partial_t(r^2\omega^\theta) + b\cdot\nabla(r^2\omega^\theta) -
3rv^r\omega^\theta = \frac{\partial_z\Gamma^2}{r} +
\Delta(r^2\omega^\theta) - \frac{4}{r}\partial_r(r^2\omega^\theta) +
3\omega^\theta.
\end{equation}
The standard energy estimate gives that
\begin{eqnarray}\nonumber
&&\frac{1}{2}\frac{d}{dt}\|r^2\omega^\theta\|_{L^2}^2 +
  \|\nabla(r^2\omega^\theta)\|_{L^2}^2\\\nonumber
&&= 3\int rv^r\omega^\theta r^2\omega^\theta dx +
  \int\partial_z\Gamma^2 r\omega^\theta dx + 3\int
  \omega^\theta r^2\omega^\theta dx.
\end{eqnarray}
It is easy to estimate that
\begin{eqnarray}\nonumber
\int\partial_z\Gamma^2 r\omega^\theta dx \leq
2\|\Gamma\|_{L^3}\|\nabla v^\theta\|_{L^2}\|r^2\omega^\theta \|_{L^6}
\leq 4\|\Gamma_0\|_{L^3}^2\|\nabla v^\theta\|_{L^2}^2 +
\frac{1}{4}\|\nabla(r^2\omega^\theta)\|_{L^2}^2.
\end{eqnarray}
Next, one also has
\begin{eqnarray}\nonumber
\int \omega^\theta r^2\omega^\theta dx &\leq&
  \|\omega^\theta\|_{L^2(r \leq R(t))}^2R^2(t) + \|r^2\omega^\theta\|_{L^2(r
  > R(t))}^2R^{-2}(t)\\\nonumber
&\leq& \|\omega^\theta\|_{L^2}^2R^2(t) +
  \|r^2\omega^\theta\|_{L^2}^2R^{-2}(t).
\end{eqnarray}
Finally, we estimate that
\begin{eqnarray}\nonumber
\int rv^r\omega^\theta r^2\omega^\theta dx &\leq&
  \|v^r\|_{L^2}\|(r^2\omega^\theta)^{\frac{3}{2}}\|_{L^4}
  \|(\omega^\theta)^{\frac{1}{2}}\|_{L^4}\\\nonumber
&\leq& \|v_0\|_{L^2}^4\|\omega^\theta\|_{L^2}^2
  + \frac{1}{4}\|\nabla(r^2\omega^\theta)\|_{L^2}^2.
\end{eqnarray}

By taking $R(t) = \|\Gamma_0\|_{L^3} + \|u_0\|_{L^2}^2$, we
arrive at
\begin{eqnarray}\nonumber
&&\frac{d}{dt}\|r^2\omega^\theta\|_{L^2}^2 +
\|\nabla(r^2\omega^\theta)\|_{L^2}^2 \lesssim (\|\Gamma_0\|_{L^3}^2
  + \|v_0\|_{L^2}^4)\|\nabla v\|_{L^2}^2\\\nonumber
&&\quad\quad\quad\quad\quad\  +\ \big(\|\Gamma_0\|_{L^3}^2 + \|u_0\|_{L^2}^4\big)^{-1}\|r^2\omega^\theta\|_{L^2}^2.
\end{eqnarray}
Clearly, \eqref{vorticity-2} follows by basic energy estimate and
applying Gronwall inequality to the above differential inequality.
\end{proof}

At last, let us prove Corollary \ref{cor} by using Theorem \ref{thm}.
\begin{proof}
It suffices to check the validity of FBC in \eqref{FBC-1}-\eqref{FBC-2} under the assumptions in Corollary. Let $\delta_0 \in (0, \frac{1}{2})$ and $C_1 > 1$ be arbitrarily large. Noting $\|\Gamma_0\|_{L^\infty} < \infty$ and using the maximum principle, we have
$$\|\Gamma(t, \cdot)\|_{L^\infty} \leq \|\Gamma_0\|_{L^\infty}.$$
Take a smooth cut-off function of $r$ such that
$$\phi \equiv 1\ {\rm if}\ 0 \leq r \leq 1,\quad \phi \equiv 0\ {\rm if}\ r \geq 2.$$

For all $\delta < \frac{\delta_0}{2}$, using \eqref{CD}, one has
\begin{eqnarray}\nonumber
\int\frac{|v^\theta|}{r}|\phi(\frac{r}{\delta}) f|^2rdrdz \leq \int\frac{C_1}{r^2|\ln r|^{2}}|\phi(\frac{r}{\delta}) f|^2rdrdz.
\end{eqnarray}
Using integration by parts, one has
\begin{eqnarray}\nonumber
&&\int\frac{1}{r^2|\ln r|^{2}}|\phi(\frac{r}{\delta}) f|^2rdrdz\\\nonumber
&&= \int|\phi(\frac{r}{\delta}) f|^2d|\ln r|^{-1}dz\\\nonumber
&&= \int|\ln r|^{-1}\phi(\frac{r}{\delta}) f\partial_r[\phi(\frac{r}{\delta}) f]drdz\\\nonumber
&&\leq \frac{1}{2}\int\frac{1}{r^2|\ln r|^{2}}|\phi(\frac{r}{\delta}) f|^2rdrdz + \frac{1}{2}\int\big|\partial_r[\phi(\frac{r}{\delta}) f]\big|^2rdrdz.
\end{eqnarray}
Hence, we have
\begin{eqnarray}\nonumber
\int\frac{1}{r^2|\ln r|^{2}}|\phi(\frac{r}{\delta}) f|^2rdrdz \leq \int\big|\partial_r[\phi(\frac{r}{\delta}) f]\big|^2rdrdz,
\end{eqnarray}
which further gives that
\begin{eqnarray}\nonumber
\int\frac{|v^\theta|}{r}|\phi(\frac{r}{\delta}) f|^2rdrdz \leq C_1\int\big|\partial_r[\phi(\frac{r}{\delta}) f]\big|^2rdrdz.
\end{eqnarray}
On the other hand, it is easy to see that
\begin{eqnarray}\nonumber
\int\frac{|v^\theta|}{r}\big|[1 - \phi(\frac{r}{\delta})] f\big|^2rdrdz \leq \|\Gamma_0\|_{L^\infty}\delta^{-2}\int\big|[1 - \phi(\frac{r}{\delta})] f\big|^2rdrdz.
\end{eqnarray}
Consequently, we have
\begin{eqnarray}\label{2-1}
&&\int\frac{|v^\theta|}{r}|f|^2rdrdz \leq 2C_1\int\big|\partial_r[\phi(\frac{r}{\delta}) f]\big|^2rdrdz\\\nonumber
&&\quad +\ 2\|\Gamma_0\|_{L^\infty}\delta^{-2}\int\big|[1 - \phi(\frac{r}{\delta})] f\big|^2rdrdz\\\nonumber
&&\leq 4C_1\int|\partial_rf|^2rdrdz +  C\delta^{-2}\int_{r \geq \delta}|f|^2rdrdz.
\end{eqnarray}
Here and in the next inequality we use $C$ to
denote a generic positive constant whose meaning may change from line to line and which may depend on $\|\Gamma_0\|_{L^\infty}$ and $C_1$.

Next, using \eqref{2-1}, we have
\begin{eqnarray}\label{2-2}
&&\int|v^\theta|^2|f|^2rdrdz \leq 2\int|rv^\theta|\frac{|v^\theta|}{r}|\phi(\frac{r}{\delta})f|^2rdrdz\\\nonumber
&&\quad +\ 2\int|rv^\theta|^2r^{-2}\big|[1 - \phi(\frac{r}{\delta})]f\big|^2rdrdz\\\nonumber
&&\leq 2C_1|\ln \delta|^{-2}\int\frac{|v^\theta|}{r}|\phi(\frac{r}{\delta})f|^2rdrdz +
 2\|\Gamma_0\|_{L^\infty}^2\delta^{-2}\int_{r \geq \delta}|f|^2rdrdz\\\nonumber
&&\leq C\delta^{-2}\int_{r \geq \delta}|f|^2rdrdz +\ 8C_1^2|\ln \delta|^{-2}\int|\partial_rf|^2rdrdz.
\end{eqnarray}
Hence, one may choose $\delta$ small enough so that $16C_1^2|\ln \delta|^{-2} \leq \delta_\ast$ and choose $C_\ast = 4C_1$. Then it is clear from \eqref{2-1} and \eqref{2-2} that the assumptions in \eqref{FBC-1}-\eqref{FBC-2} are satisfied. Using Theorem \ref{thm}, one concludes that $v$ is smooth for all $t > 0$.
\end{proof}

\section{Criticality of Axis-Symmetric Navier-Stokes Equations}

\begin{proof}[Proof of Theorem \ref{thm}]
First of all, for initial data $v_0 \in H^{\frac{1}{2}}$, by the classical results of Leray \cite{Leray} and Fujita-Kato \cite{FK}, there exists a unique local strong solution $v$ to the Navier-Stokes equations \eqref{ANS}. Moreover, $v(t, \cdot) \in H^s$ for any $s \geq 0$ at least on a short time interval $[\epsilon, 2\epsilon]$. In particular, $\nabla\omega(t, \cdot) \in L^2$ at least on a short time interval $[\epsilon, 2\epsilon]$.
A consequence is that $\nabla \omega^r$, $\nabla \omega^\theta$, $\nabla \omega^z$, $\frac{\omega^r}{r}$, $\frac{\omega^\theta}{r}$ are all $L^2$-functions. In particular, recalling that
$$J = \frac{\omega^r}{r},\quad \Omega = \frac{\omega^\theta}{r},$$
one has
$J(t, \cdot) \in L^2,\quad \Omega(t, \cdot) \in L^2$ for $t \in [\epsilon, 2\epsilon]$. Inductively, one also has $J(t, \cdot) \in H^2$ and $\Omega(t, \cdot) \in H^2$. Without loss of generality, we may assume that
$$J_0 \in H^2,\quad \Omega_0 \in H^2.$$
Otherwise we may start from $t = \epsilon$. As long as the solution is still smooth, one has
$$\|J(t, \cdot)\|_{L^2} + \|\Omega(t, \cdot)\|_{L^2} < \infty,\quad \|\nabla J(t, \cdot)\|_{L^2}^2 + \|\nabla \Omega(t, \cdot)\|_{L^2}^2  < \infty$$
and
$$\int_{-\infty}^\infty\big(|J(t, 0, z)|^2 + |\Omega(t, 0, z)|^2\big)dz \lesssim \|J(t, \cdot)\|_{H^2}^2 + \|\Omega(t, \cdot)\|_{H^2}^2 < \infty.$$
So all calculations are legal below as long as the solution is still smooth. Our task is to derive certain strong enough \textit{a priori} estimate.

By applying standard energy estimate to $J$ equation, we have
\begin{eqnarray}\nonumber
\frac{1}{2}\frac{d}{dt}\|J\|_{L^2}^2 &=& - \int J(b\cdot\nabla) Jrdrdz + \int J(\Delta  + \frac{2}{r}\partial_r)J \\\nonumber
&&+\ \int J(\omega^r\partial_r + \omega^z\partial_z)\frac{v^r}{r}rdrdz.
\end{eqnarray}
Using the incompressibility constraint, one has
\begin{eqnarray}\nonumber
- \int J(b\cdot\nabla) Jrdrdz = \frac{1}{2}\int J^2\nabla\cdot b rdrdz = 0.
\end{eqnarray}
On the other hand, by direct calculations, one has
\begin{eqnarray}\nonumber
\int J(\Delta  + \frac{2}{r}\partial_r)J = - \|\nabla J\|_{L^2}^2 - \int |J(t, 0, z)|^2dz.
\end{eqnarray}
Consequently, we have
\begin{eqnarray}\label{E3}
&&\frac{1}{2}\frac{d}{dt}\|J\|_{L^2}^2
+ \|\nabla J\|_{L^2}^2 +\ \int_{-\infty}^\infty|J(t, 0, z)|^2dz\\\nonumber
&&\quad\quad\quad = \int J(\omega^r\partial_r + \omega^z\partial_z)\frac{v^r}{r}rdrdz
\end{eqnarray}
Similarly, by applying the energy estimate to the equation of $\Omega$, one obtains that
\begin{eqnarray}\label{E4}
&&\frac{1}{2}\frac{d}{dt}\|\Omega\|_{L^2}^2
+ \|\nabla \Omega\|_{L^2}^2 + \int_{-\infty}^\infty|\Omega(t, 0, z)|^2dz\\\nonumber
&&\quad\quad\quad\quad\quad =  - 2\int\frac{v^\theta}
  {r}J\Omega rdrdz.
\end{eqnarray}

In the remaining part of the proof of Theorem \ref{thm}, we will use $C$ to
denote a generic positive constant whose meaning may change from line to line and which may depend on $\|\Gamma_0\|_{L^\infty}$, $C_0$, $C_\ast$ and $r_0$.  Using $\|\Gamma\|_{L^\infty} \leq \|\Gamma_0\|_{L^\infty}$ and the Form Boundedness Condition in \eqref{FBC-1}, one has
\begin{eqnarray}\label{E1}
\Big|\int\frac{v^\theta}
  {r}J\Omega rdrdz\Big| &\leq& \frac{1}{4C_\ast}\int\big|\frac{v^\theta}
  {r}\big|\Omega^2 rdrdz + C_\ast\int\big|\frac{v^\theta}
  {r}\big| J^2 rdrdz\\\nonumber
&\leq& \frac{1}{4}\int|\partial_r\Omega|^2rdrdz + C_\ast^2\int|\partial_rJ|^2rdrdz\\\nonumber
&&+\ C\int_{r \geq r_0}(|J|^2 + |\Omega|^2)rdrdz.
\end{eqnarray}
Inserting \eqref{E1} into \eqref{E4}, one has
\begin{eqnarray}\label{E5}
&&\frac{d}{dt}\|\Omega\|_{L^2}^2
+ \|\nabla \Omega\|_{L^2}^2 + 2\int_{-\infty}^\infty|\Omega(t, 0, z)|^2dz\\\nonumber
&&\leq 2C_\ast^2\|\nabla J\|_{L^2}^2 + C\|\omega\|_{L^2}^2.
\end{eqnarray}

Next, we estimate that
\begin{eqnarray}\nonumber
&&\Big|\int J(\omega^r\partial_r + \omega^z\partial_z)\frac{v^r}{r}dx\Big|\\\nonumber
&&= \Big|\int [\nabla\times(v^\theta e_\theta)]\cdot(J\nabla\frac{v^r}{r})dx\Big|\\\nonumber
&&\leq \|\nabla J\|_{L^2}\|v^\theta\nabla\frac{v^r}{r}\|_{L^2}.
\end{eqnarray}
Again, using the Form Boundedness Condition in \eqref{FBC-2}, one has
$$\|v^\theta\nabla\frac{v^r}{r}\|_{L^2}^2 \leq \delta_\ast\|\partial_r\nabla\frac{v^r}{r}\|_{L^2}^2 + C_0\int_{r \geq r_0}|\nabla\frac{v^r}{r}|^2drdrdz.$$
Using Lemma \ref{Lem1} and the identity
$$\nabla\frac{v^r}{r} = \frac{\nabla v^r}{r} - e_r\frac{v^r}{r^2} = \frac{\nabla v^r}{r} + e_r\frac{\partial_rv^r + \partial_zv^z}{r},$$ we have
\begin{eqnarray}\label{E2}
&&\Big|\int J(\omega^r\partial_r + \omega^z\partial_z)\frac{v^r}{r}rdrdz\Big|\\\nonumber
&&\leq \frac{1}{2}\|\nabla J\|_{L^2}^2 + \frac{\delta_\ast}{2}\|\partial_r\nabla\frac{v^r}{r}\|_{L^2}^2 + \frac{C_0}{2}\int_{r \geq r_0}|\nabla\frac{v^r}{r}|^2drdrdz\\\nonumber
&&\leq \frac{1}{2}\|\nabla J\|_{L^2}^2 + \frac{K_0\delta_\ast}{2}\|\partial_z\Omega\|_{L^2}^2 + C\int|\nabla v|^2drdrdz\\\nonumber
\end{eqnarray}
Inserting \eqref{E2} into \eqref{E3}, we have
\begin{eqnarray}\label{E6}
&&\frac{d}{dt}\|J\|_{L^2}^2
+ \|\nabla J\|_{L^2}^2 + 2\int_{-\infty}^\infty|J(t, 0, z)|^2dz\\\nonumber
&&\leq K_0\delta_\ast\|\partial_z\Omega\|_{L^2}^2 + C\int|\nabla v|^2drdrdz.
\end{eqnarray}

Multiplying \eqref{E6} by $3C_\ast^2$ and then adding it up with \eqref{E5}, we have
\begin{eqnarray}\nonumber
&&\frac{d}{dt}\big(3C_\ast^2\|J\|_{L^2}^2 + \|\Omega\|_{L^2}^2
\big)
+ \big(C_\ast^2\|\nabla J\|_{L^2}^2 + \|\nabla \Omega\|_{L^2}^2
\big)\\\nonumber
&&\quad +\ \int_{-\infty}^\infty\big(6C_\ast^2|J(t, 0, z)|^2 + 2|\Omega(t, 0, z)|^2\big)dz\\\nonumber
&&\leq K_0\delta_\ast\|\nabla\Omega\|_{L^2}^2 + C\int|\nabla v|^2rdr
dz.
\end{eqnarray}
Integrating the above inequality with respect to time, we have
\begin{eqnarray}\nonumber
&&3C_\ast^2\|J(t, \cdot)\|_{L^2}^2 + \|\Omega(t, \cdot)\|_{L^2}^2
+ \int_0^t\big(C_\ast^2\|\nabla J\|_{L^2}^2 + \|\nabla \Omega\|_{L^2}^2
\big)ds\\\nonumber
&&\quad +\ \int_0^t\int_{-\infty}^\infty\big(6C_\ast^2|J(t, 0, z)|^2 + 2|\Omega(t, 0, z)|^2\big)dzds\\\nonumber
&&\leq CC_\ast^2(\|J_0\|_{L^2}^2 + \|\Omega_0\|_{L^2}^2) + K_0\delta_\ast\int_0^t\|\nabla\Omega\|_{L^2}^2ds  +\ C\|v_0\|_{L^2}.
\end{eqnarray}
Here we used the basic energy identity \eqref{NI}. Recall that $K_0$ is an absolute positive constant determined in Lemma \ref{Lem1}. Hence, we may take $\delta_\ast$ so that
$K_0\delta_\ast < \frac{1}{2}$.
Consequently, we have
\begin{eqnarray}\label{E8}
&&\sup_{0 \leq t < T}\big(\|J(t)\|_{L^2}^2 + \|\Omega(t)\|_{L^2}^2
\big)
+ \int_0^T\big(\|\nabla J\|_{L^2}^2 + \|\nabla \Omega\|_{L^2}^2
\big)dt\\\nonumber
&&\quad +\ \int_0^T\int_{-\infty}^\infty\big(|J(t, 0, z)|^2 + |\Omega(t, 0, z)|^2\big)dzdt < \infty
\end{eqnarray}
for all $T < \infty$.

Clearly, the \textit{a priori} estimate \eqref{E8} and Sobolev imbedding theorem imply that
\begin{equation}\label{7-3}
\sup_{0 \leq t < T}\|b(t, \cdot)\|_{L^p(r \leq 1)} < \infty,\ \ 3 \leq p \leq 6,\quad T < \infty.
\end{equation}

\begin{rem}
There are several ways to prove regularity of $v$ from here. For instance, the easiest way is just to use the result in \cite{NP}.
If one would like to assume further decay properties on the initial data so that the conditions in Lemma \ref{decay} are satisfied, then one can easily derive an $L^\infty([0, T], L^p)$ estimate for $b$ when $r \geq 1$ and $3 \leq p \leq 6$, by using the basic energy estimate and Sobolev imbedding theorem. This combining the local $L^p$ estimate of $v$ in \eqref{7-3} immediate gives that $b \in L^\infty([0, T], L^3)$ for any $T < \infty$. Using Sobolev imbedding theorem once more, one has $b \in L^\infty([0, T], {\rm BMO}^{-1})$. Then  the result in \cite{LZ} implies that $v$ is regular up to time $T$.
\end{rem}

Let us give an alternative and simple self-contained proof below. We first derive an $L^4$ \textit{a priori} estimate for $v^\theta$ without using the result of \cite{LZ} and Lemma \ref{decay}. Using the equation of $v^\theta$ in \eqref{ANS} and standard energy estimate, one has
\begin{eqnarray}\nonumber
&&\frac{d}{dt}\|v^\theta\|_{L^4}^4 + \|\nabla(v^\theta)^2\|_{L^2}^2 + \big\|r^{-1}(v^\theta)^2\big\|_{L^2}^2\\\nonumber
&&\leq C\Big|\int\frac{v^r(v^\theta)^4}{r}rdrdz\Big|\\\nonumber
&&\leq C\|r^{-1}v^r\|_{L^\infty}\|v^\theta\|_{L^4}^4.
\end{eqnarray}
Hence, by using Lemma \ref{Lem1} and three dimensional interpolation inequality
$\|f\|_{L^\infty}^2 \lesssim \|\nabla f\|_{L^2}\|\nabla^2
f\|_{L^2}$, one has
\begin{eqnarray}\nonumber
\big\|\frac{v^r}{r} \big\|_{L^{\infty}} &\lesssim&
\big\|\nabla\partial_z\frac{\psi^\theta}{r}\big\|_{L^2}
^{\frac{1}{2}}\big\|\nabla^2\partial_z\frac{\psi^\theta}{r}\big\|_{L^2}
^{\frac{1}{2}}\\\nonumber
&=& \big\|\nabla\frac{v^r}{r}\big\|_{L^2}
^{\frac{1}{2}}\big\|\nabla^2\frac{v^r}{r}\big\|_{L^2}
^{\frac{1}{2}} \lesssim
\|\Omega\|_{L^2}^{\frac{1}{2}}\|\partial_z\Omega\|_{L^2}
^{\frac{1}{2}}.
\end{eqnarray}
and \eqref{E8}, one concludes from Gronwall's inequality that
$$\|v^\theta\|_{L^4} < \infty,\quad \int_0^T\big\|r^{-1}(v^\theta)^2\big\|_{L^2}^2dt < \infty,\quad 0 \leq t \leq T.$$
Then we use the equation of $\omega^\theta$ to derive that
\begin{eqnarray}\nonumber
&&\frac{d}{dt}\|\omega^\theta\|_{L^2}^2 + 2\|\nabla\omega^\theta\|_{L^2}^2 + 2\|r^{-1}\omega^\theta\|_{L^2}^2\\\nonumber
&&= - \int\frac{v^r}{r}(\omega^\theta)^2rdrdz + \int\omega^\theta\partial_z\frac{(v^\theta)^2}{r}rdrdz\\\nonumber
&&\leq C\|r^{-1}v^r\|_{L^\infty}\|\omega^\theta\|_{L^2}^2 + \|\partial_z\omega^\theta\|_{L^2}^2 + \frac{1}{4}\big\|r^{-1}(v^\theta)^2\big\|_{L^2}^2.
\end{eqnarray}
Hence, Gronwall's inequality similarly gives that
$$\omega^\theta \in L^2,\quad T < \infty.$$
By basic energy identity \eqref{NI} and Sobolev imbedding, one has
$$b \in L^p,\quad 2 \leq p \leq 6.$$
Hence, $v \in L^\infty_T(L^4_x)$. So Serrin type criterion implies that $v$ is regular up to time $T$.

\end{proof}

\section{Small $\|\Gamma_0\|_{L^\infty}$ or $\|\Gamma(t ,\cdot)\|_{L^\infty(r \leq r_0)}$ Global Regularity}

This section is devoted to proving Theorem \ref{SmallnessRegu}.
\begin{proof}[Proof of Theorem \ref{SmallnessRegu}]
Recall that
$$V = \frac{v^\theta}{\sqrt{r}},\quad \Omega =
\frac{\omega^\theta}{r}.$$
Let us first formulate the axi-symmetric Navier-Stokes equations \eqref{ANS} in terms of $V$ and $\Omega$ as follows:
\begin{equation}\label{WeiEqn}
\begin{cases}
\partial_tV + b\cdot\nabla V + \frac{3v^r}{2r}V
  = \big(\Delta + \frac{1}{r}\partial_r - \frac{3}{4r^2}\big)V,\\[-4mm]\\
\partial_t\Omega + b\cdot\nabla \Omega
  = \big(\Delta + \frac{2}{r}\partial_r \big)\Omega + \frac{2\partial_zV^2}{r}.
\end{cases}
\end{equation}
By energy estimate, one has
\begin{equation}\label{7-1}
\frac{1}{2}\frac{d}{dt}\|\Omega\|_{L^2}^2 + \|\nabla\Omega\|_{L^2}^2
\leq \frac{1}{2}\|\partial_z\Omega\|_{L^2}^2 +
\frac{1}{2}\big\|r^{-1}|V|^2\big\|_{L^2}^2.
\end{equation}

Let us first prove the global regularity under
$$\|\Gamma_0\|_{L^\infty} \leq \delta M_0^{-1}.$$
Using Lemma \ref{Lem1}, one has
\begin{eqnarray}\nonumber
\big\|\frac{v^r}{r} \big\|_{L^{\infty}}  \lesssim
\|\Omega\|_{L^2}^{\frac{1}{2}}\|\partial_z\Omega\|_{L^2}
^{\frac{1}{2}}.
\end{eqnarray}
Noting that
\begin{eqnarray}\nonumber
\|V\|_{L^4}^4 \lesssim
\big\|r^{-1}|V|^2\big\|_{L^2}^{\frac{3}{2}}\|\Gamma\|_{L^4},
\end{eqnarray}
one can apply $L^4$ energy estimate for $V$ to get
\begin{eqnarray}\label{7-2}
&&\frac{d}{dt}\big\||V|^2\big\|_{L^2}^2 + \big\|\nabla
  |V|^2\big\|_{L^2}^2 + \big\|r^{-1}|V|^2\big\|_{L^2}^2\\\nonumber
&&\lesssim \big\|\frac{v^r}{r} \big\|_{L^{\infty}}
  \|V\|_{L^4}^4 \lesssim \|\Omega\|_{L^2}^{\frac{1}{2}}
  \|\partial_z\Omega\|_{L^2}^{\frac{1}{2}}\big\|r^{-1}|V|^2
  \big\|_{L^2}^{\frac{3}{2}}\|\Gamma\|_{L^4}\\\nonumber
&&\lesssim \|\Omega\|_{L^2}^{\frac{1}{2}}\|\Gamma\|_{
  L^2}^{\frac{1}{2}}\|\Gamma\|_{L^\infty}^{\frac{1}{2}}
  \big(\|\partial_z\Omega\|_{L^2}^2+ \big\|r^{-1}
  |V|^2\big\|_{L^2}^2\big).
\end{eqnarray}

Combining \eqref{7-1} and \eqref{7-2}, we arrive at
\begin{eqnarray}\label{S7-7}
&&\frac{d}{dt}\Big(\big\||V|^2\big\|_{L^2}^2 +
  \|\Omega\|_{L^2}^2\Big) + \Big(\big\|\nabla|V|^2
  \big\|_{L^2}^2 + \|\nabla\Omega\|_{L^2}^2\Big)\\\nonumber
&&\lesssim \|\Omega\|_{L^2}^{\frac{1}{2}}\|\Gamma\|_{
  L^2}^{\frac{1}{2}}\|\Gamma\|_{L^\infty}^{\frac{1}{2}}
  \big(\|\partial_z\Omega\|_{L^2}^2+ \big\|r^{-1}
  |V|^2\big\|_{L^2}^2\big).
\end{eqnarray}
Recall that we have the following \textit{a priori} estimate:
\begin{equation}\nonumber
\|\Gamma\|_{L^2} \leq \|\Gamma_0\|_{L^2},\quad \|\Gamma\|_{L^\infty}
\leq \|\Gamma_0\|_{L^\infty}.
\end{equation}
Hence, under the condition of the theorem, there exists $T > 0$ such
that
$$\big\||V|^2\big\|_{L^2}^2 +
\|\Omega\|_{L^2}^2 < 2\big\||V_0|^2\big\|_{L^2}^2 +
2\|\Omega_0\|_{L^2}^2,\quad \forall\ \ 0 \leq t < T.$$ If  $\delta$ is a suitably small positive constant and
$$\|\Gamma_0\|_{L^\infty} \leq \delta M_0^{-1}$$
is satisfied, then we have
\begin{eqnarray}\nonumber
\|\Omega\|_{L^2} ^{\frac{1}{2}}\|\Gamma\|_{L^2}^{
  \frac{1}{2}}\|\Gamma\|_{L^\infty}^{\frac{1}{2}}
\lesssim M_0^{\frac{1}{2}}(\delta M_0^{- 1})^{\frac{1}{2}} \lesssim
  \delta^{\frac{1}{2}},\quad \forall\ \ 0 \leq t < T,
\end{eqnarray}
Hence, by \eqref{S7-7}, we derive that
\begin{eqnarray}\nonumber
\frac{d}{dt}\Big(\big\||V|^2\big\|_{L^2}^2 +
  \|\Omega\|_{L^2}^2\Big) \leq 0, \quad \forall\ \ 0 \leq t <
  T,
\end{eqnarray}
which implies that
\begin{eqnarray}\nonumber
\big\||V|^2\big\|_{L^2}^2 + \|\Omega\|_{L^2}^2 \leq
\big\||V_0|^2\big\|_{L^2}^2 + \|\Omega_0\|_{L^2}^2, \quad 0 \leq t
\leq T.
\end{eqnarray}
The above argument implies, by the standard continuation method  that
\begin{eqnarray}\nonumber
\big\||V|^2\big\|_{L^2}^2 + \|\Omega\|_{L^2}^2 \leq
\big\||V_0|^2\big\|_{L^2}^2 + \|\Omega_0\|_{L^2}^2, \quad \forall\ \
t \geq 0.
\end{eqnarray}
Hence the proof for first part of the theorem is finished.

On the other hand, if
$$\|\Gamma(t, \cdot)\|_{L^\infty(r \leq r_0)} \leq \delta M_1^{-1}$$
is satisfied, then  one may treat \eqref{7-2} as follows:
\begin{eqnarray}\nonumber
&&\frac{d}{dt}\big\||V|^2\big\|_{L^2}^2 + \big\|\nabla
  |V|^2\big\|_{L^2}^2 + \big\|r^{-1}|V|^2\big\|_{L^2}^2\\\nonumber
&&\lesssim \big\|\frac{v^r}{r} \big\|_{L^{\infty}}
  \|V\|_{L^4(r \leq r_0)}^4 + \int_{r \geq r_0}\big|\frac{v^r}{r}\frac{(v^\theta)^4}{r^2}\big|rdrdz\\\nonumber
&&\lesssim \|\Omega\|_{L^2}^{\frac{1}{2}}\|\Gamma\|_{
  L^2}^{\frac{1}{2}}\|\Gamma\|_{L^\infty(r \leq r_0)}^{\frac{1}{2}}
  \big(\|\partial_z\Omega\|_{L^2}^2+ \big\|r^{-1}
  |V|^2\big\|_{L^2}^2\big)\\\nonumber
&&\quad +\ r_0^{-4}\big\|\frac{v^r}{r}\big\|_{L^2}\big\|\frac{v^\theta}{r}\big\|_{L^2}\|\Gamma\|_{L^\infty(r \geq r_0)}^3,
\end{eqnarray}
which combining \eqref{7-1} gives that
\begin{eqnarray}\nonumber
&&\big\||V(t, \cdot)|^2\big\|_{L^2}^2 +
  \|\Omega(t, \cdot)\|_{L^2}^2 + \int_0^t\Big(\big\|\nabla|V|^2
  \big\|_{L^2}^2 + \|\nabla\Omega\|_{L^2}^2\Big)ds\\\nonumber
&&\lesssim \big\||V_0|^2\big\|_{L^2}^2 +
  \|\Omega_0\|_{L^2}^2 + r_0^{-4}\|v_0\|_{L^2}^2\|\Gamma_0\|_{L^\infty}^3\\\nonumber
&&\quad +\ \big(\|\Gamma\|_{L^\infty(r \leq r_0)}\|\Gamma_0\|_{
  L^2}\sup_{0 \leq s \leq t}\|\Omega(s, \cdot)\|_{L^2}\big)^{\frac{1}{2}}
  \big(\|\partial_z\Omega\|_{L^2}^2+ \big\|r^{-1}
  |V|^2\big\|_{L^2}^2\big).
\end{eqnarray}
Here $r_0 > 0$ is arbitrary.  Then similar continuation arguments as the proof used in the first part implies that if $\delta$ is a suitable small absolute positive constant, then the solution $v$ is regular. Here $M_1$ is given in the statement of Theorem \ref{SmallnessRegu}. This shows that the smallness of $\Gamma$ locally in $r$ implies the regularity of the solutions.
\end{proof}

\begin{rem}\nonumber
The constant $M_0$ in Theorem \ref{SmallnessRegu} is dimensionless.
Indeed, $\big\||V_0|^2\big\|_{L^2}$ and $\|\Omega_0\|_{L^2}$ have
dimension $- \frac{3}{2}$, and $\|\Gamma_0\|_{L^2}$ has dimension
$\frac{3}{2}$. Hence, the dimension of $M_0$ is 0. Similarly, one can also check that $M_1$ is also dimensionless if one assigns $r_0$ dimension 1.
\end{rem}

\section*{Acknowledgement}
The authors would like to thank Prof. Daoyuan Fang, Dr. Bin Han, Mr. Xinghong Pan and Prof. Ting Zhang for discussions. Z. Lei was in part
supported by NSFC (grant No.11171072, 11421061 and 11222107), 
Shanghai Shu Guang project, Shanghai Talent Development Fund and SGST 09DZ2272900. Qi S. Zhang is grateful of the supports by the Siyuan Foundation through Nanjing University and by the Simons foundation.

\end{document}